\documentclass{amsart}

\newtheorem{theorem}{Theorem}
\newtheorem{corollary}{Corollary}
\newtheorem{definition}{Definition}
\newtheorem{lemma}{Lemma}
\newcommand{\C}{\mathbb C}
\newcommand{\R}{\mathbb R}
\newcommand{\Z}{\mathbb Z}

\title{Conformal extensions of functions defined on arbitrary subsets of Riemann surfaces}

\author{P. M. Gauthier and V. Nestoridis}

\address{D\'epartement de math\'ematiques et de statistique, Universit\'e de Montr\'eal,
CP-6128 Centreville, Montr\'eal,  H3C3J7, CANADA}
\email{gauthier@dms.umontreal.ca}

\address{Department of Mathematics, University of Athens Panepisitmioupolis, 15784 Athens, Greece}
\email{vnestor@math.uoa.gr}

\keywords{analytic continuation, analytic arc} \subjclass[2010]{Primary: 30B40; Secondary: 30F99 }

\thanks{Research supported by NSERC (Canada)}

\begin{document}

\begin{abstract}
For a function defined on an arbitrary subset of a Riemann surface, we give conditions which allow the function to be extended conformally One folkloric consequence is that two common definitions of an analytic arc in $\C$ are equivalent.
\end{abstract}

\maketitle

The purpose of this note is to extend conformally a function given on an arbitrary subset of a Riemann surface. Our original motivation was to prove that a construction of Nestoridis and Zadik in \cite{NZ} holds more generally, so there is no need to use the specific form of the extended function used in \cite[Prop. 2.3, iii]{NZ}. A corollary of our result is that two common definitions of analytic arcs in $\C$ are equivalent.

\begin{definition}
A (topological or Jordan) open arc $J$ in a topological space $X$ is a subset $J\subset X,$ which is a homeomorphic image of the open unit interval
$$
    I = \{t\in\R: 0<t<1\},
$$
equivalently, of the real line.
Thus, if $J$ is an open arc, there is a parametrization $\varphi:I\rightarrow J,$ which is a homeomorphism.
\end{definition}

If we are considering an open arc in a Riemann surface $X,$ then topological notions (closure, boundary, etc.) will be with respect to $X.$ In particular, if we are considering an open arc in the Riemann sphere $\overline\C,$ then topological notions will be with respect to $\overline\C.$ For a holomorphic mapping $f:X\rightarrow Y$ between two Riemann surfaces and a point $p\in X,$ we write $f^\prime(p)=0$ to signify that, for some (hence any) local coordinate mappings $\varphi$ and $\psi$ at $p$ and $f(p)$ respectively, with $\varphi(p)=0,$ we have $(\psi\circ f\circ\varphi^{-1})^\prime(0)=0.$ In particular, either $X$ or $Y$ may be the Riemann sphere $\overline\C.$

For an open arc $J$ in $X$  with parametrization $\varphi,$
we define the initial end $J(0)$ and the terminal end $J(1)$ of $J$ as
$$
    J(0)=\bigcap_{0<t<1}\overline{\varphi(0,t]}, \quad J(1) = \bigcap_{0<t<1}\overline{\varphi[t,1)}.
$$
Since $\varphi$ is a homeomorphism onto $J$ and $J$ has the relative topology induced by $X,$ it follows that the open arc $J$ is disjoint from both of its ends.
Each end is a closed connected set. If the initial end is a point, we call this the initial point of the arc (even though it is not on the arc). Similarly, if the terminal end is a point, we call it the terminal point.

If $\varphi:I\rightarrow J$ is an open arc and $f:J\rightarrow\overline\C,$  we define the initial and terminal cluster sets of $f$ on $J:$
$$
    C_0(f,J) = \bigcap_{t\in I}\overline{\{f(\varphi(s)):0<s<t\}}, \quad
    C_1(f,J) = \bigcap_{t\in I}\overline{\{f(\varphi(s)):t<s<1\}}.
$$
The cluster set $C(f,J)$ is the union of the initial cluster set of $f$ and the terminal cluster set of $f.$ In other words,
$$
    C(f,J) = C_0(f,J)\cup C_1(f,J) = \bigcap_{\epsilon>0}\overline{(f\circ\varphi)(I\setminus[\epsilon,1-\epsilon])}.
$$

\begin{theorem}
Suppose $J$ is an open arc in $\overline\C$ and $f:U\rightarrow\overline\C$ is a holomorphic mapping on
an open neighborhood $U$ of $J$ in $\overline\C.$ Suppose $f|_J$ is injective, $f^\prime(z)\not=0,$ for $z\in J,$ and the sets $f(J)$  and $C(f,J)$ are disjoint. Then, $f$ is injective (one-to-one conformal) on some neighborhood of $J.$
\end{theorem}

\begin{proof}
Assume, first of all, that the initial and terminal cluster sets $C_o(f,J)$ and $C_1(f,J)$ are disjoint.
Consider the mapping $\psi:I\rightarrow f(J),$ given as
$\psi=f\circ\varphi.$ As a composition of continuous injective mappings, $\psi$ is also continuous and injective. We claim that $\psi^{-1},$ (which is well-defined) is also continuous. Suppose for the sake of contradiction, that there is a sequence $\psi(t_j), t_j\in I,$ which converges to a point $\psi(\alpha), \alpha\in I,$ but $t_j\not\rightarrow\alpha.$ We may assume that $t_j$ converges to point $\beta\in[0,1].$ If $\beta\in I,$ then $\psi(t_j)\rightarrow\psi(\beta)\not=\psi(\alpha),$ which is a contradiction. If $\beta=0,$ then $t_j\rightarrow 0,$ so $\psi(\alpha)=\lim\psi(t_j)\in C_o(f,J),$ since $t_j\rightarrow 0,$ which again is a contradiction, since $f(J)$ is disjoint from $C_o(f,J).$ The same argument shows that $t_j$ cannot converge to $1.$ Thus, $\psi^{-1}$ is continuous and so $\psi$ is a homeomorphism. This shows that $f(J)$ is also an open arc.

Let $W_o$ be the component of $\overline\C\setminus C_o(f,J)$ which contains the connected set $f(J)\cup C_1(f,J)$ and let $W_1$ be the component of $\overline\C\setminus C_1(f,J)$ which contains the connected set $f(J)\cup C_o(f,J).$
Let $\widehat C_o(f,J)$ be the union of $C_o(f,J)$ with all of its complementary components in $\overline\C$ which do not meet the connected set $f(J)\cup C_1(f,J).$ We define $\widehat C_1(f,J)$ similarly. We may also say that $\widehat C_o(f,J)=\overline\C\setminus W_o$ and $\widehat C_1(f,J)=\overline\C\setminus W_1.$

The sets $f(J),$ $\widehat C_o(f,J)$ and $\widehat C_1(f,J)$ are pairwise disjoint. Both compact connected sets $\widehat C_o(f,J)$ and $\widehat C_1(f,J)$ have only one complementary component $W_o$ and $W_1$ respectively in $\overline\C$ which both contain $f(J).$
There exists a homeomorphism
$$
    h:\overline \C\setminus[\widehat C_o(f,J)\cup\widehat C_1(f,J)]\rightarrow\overline\C\setminus\{p_o,p_1\},
$$
mapping the topological annulus $\overline C\setminus[\widehat C_o(f,J)\cup\widehat C_1(f,J)]$ onto the twice punctured sphere $\overline\C\setminus\{p_o,p_1\},$ where $p_o$ and $p_1$ are distinct finite points. The homeomorphism $h$ maps the open arc $f(J)$ to the open arc $h(f(J)),$ whose initial and terminal points are respectively $p_o$ and $p_1.$
By composing with a M\"obius transformation, we may assume that the open arc $h(f(J))$ does not pass through $\infty.$
Let $\alpha_o$ and $\alpha_1$ be disjoint open arcs in $\C,$ where $\alpha_o$ joins $\infty$ to $p_o$ and $\alpha_1$ joins $p_1$ to infinity. Set
$$
    \alpha = \alpha_o\cup\{p_o\}\cup h(f(J)) \cup\{p_1\}\cup\alpha_1.
$$
The open arc $h(f(J))$  is the homeomorphic image of the open unit interval $I$ under the parametrization $h\circ f\circ\varphi$ with initial point $p_o$ and terminal point $p_1.$ We may extend this to a parametrization $\eta:(-\infty,+\infty)\rightarrow\alpha$ of the open arc $\alpha.$ By the Schoenflies Theorem
 \cite[page 81]{Gr}, we may further extend $\eta$ to a homeomorphism $\eta:\C\rightarrow\C.$ Let us denote the homeomorphism
 $$
    h^{-1}\circ\eta:\C\setminus\{0,1\}\rightarrow\C\setminus[\widehat C_o(f,J)\cup\widehat C_1(f,J)]
 $$
 by $H.$

Fix $t\in I.$ We may choose a closed disc $\overline D_t$  with center $t$ so small that $0,1\not\in\overline D_t$ and so $H$ maps $D_1$ homeomorphically onto a Jordan domain $W_t$ containing $f(\varphi(t)).$ Since  $f^\prime(\varphi(t))\not=0,$ we may choose a branch $g_t$ of $f^{-1}$ in a neighborhood of $f(\varphi(t)),$ such that $f^{-1}\circ f$ is the identity in a neighborhood of $\varphi(t).$ We may choose $D_t$ so small that the Jordan domain $H(D_t)$ is contained in the domain of definition of this inverse branch $g_t.$ We claim that these inverse branches, for various $t$ are compatible. Indeed, since $H$ is a homeomorphism, two Jordan domains $H(D_s)$ and $H(D_t)$ intersect if and only if the discs $D_s$ and $D_t$ intersect and in this case the intersection $H(D_s)\cap H(D_t)$ has only one component, which is $H(D_s\cap D_t).$ Since $f$ is injective on $J,$ the branches $g_s$ and $g_t$ agree on the (non-empty) arc of $f(J)$ in $H(D_s)\cap H(D_t).$ By the uniqueness principle, $g_s=g_t$ on $H(D_s)\cap H(D_t).$ We have verified that the inverse branches $g_t, t\in I,$ are compatible. Thus, we may define a branch $g$ of $f^{-1}$ on the neighborhood $W=\cup_{t\in I}H(D_t)$ of $f(J).$ We have that $f$ maps the open neighborhood $g(W)$ of $J$ biholomorphically onto the neighborhood $W$ of $f(J).$ This completes the proof, in case the initial and terminal cluster sets $C_o(f,J)$ and $C_1(f,J)$ are disjoint.

Suppose the initial and terminal cluster sets $C_o(f,J)$ and $C_1(f,J)$ are not disjoint.
The cluster set $C(f,J)$ is then a continuum or a point, so the open set $\overline\C\setminus C(f,J)$ is simply connected. In particular,  the component $\Omega$ of $\overline\C\setminus C(f,J)$ which contains the connected set $f(J)$ is simply connected. Set $E=\overline\C\setminus\Omega.$ If $E$ is a singleton, we may map $\Omega$ to $\C$ by a M\"obius transformation so that $E$ goes to $\infty.$
Suppose $E$ is a continuum. By the Riemann mapping theorem, we may assume that $\Omega$ is the unit disc and $f(J)$ is an open arc in $\Omega.$ There is a homeomorphism $h:\Omega\rightarrow\C$ so that $h(w)\rightarrow\infty,$ as $|w|\nearrow 1.$ Thus, whether $E$ is a singleton or not, there is a homeomorphism from $\Omega$ to $\C.$    After this mapping, the image of $f(J)$ is an open arc both ends of which are $\infty.$
We may parametrize $f(J)$ by the real line and use the Schoenflies theorem as above.
\end{proof}

This result extends to arcs on Riemann surfaces. For that purpose, the following lemma is useful and of independent interest.

\begin{lemma}
Let $J$ be an open arc in a Riemann surface $X.$ Then, $J$ has a fundamental system of simply connected neighborhoods.
\end{lemma}

\begin{proof}
For simplicity of notation, it will be convenient to consider $J$ as being parametrized by $\R$ rather than $(0,1).$ Thus, $\varphi:(-\infty,+\infty)\rightarrow J.$ We may choose an increasing  sequence $t_j, j\in\Z,$ such that $\lim_{j\rightarrow\pm\infty}=\pm\infty$ and each  $\varphi(t_j,t_{j+1})=J_j$ is contained in a chart $U_j.$ We may construct Jordan domains $V_j,$ such that $J_j\subset V_j\subset U_j,$ the end points of $J_j$ are on $\partial V_j$ and the $\overline V_j$ are disjoint except possibly for end points. By construction, the set $V=\cup_j V_j$ covers $J\setminus\cup_j\{\varphi(t_j)\}.$ By introducing suitable small neighborhoods of the end points $\varphi(t_j),$ we may enlarge $V$ to a ``strip" $S$ which is homeomorphic to $\{z=x+iy:-\infty<x<+\infty, |y|<1\}$ and hence simply connected. If $W$ is a neighborhood of $J,$ we can replace $X$ by the component of $W$ containing $J.$ We have thus shown that $J$ has a fundamental system of simply connected neighborhoods.
\end{proof}

\begin{theorem}
Let $J$ be an open arc in a Riemann surface $X$ and $f:U\rightarrow\overline\C$ a holomorphic mapping on
an open neighborhood $U$ of $J$ in $X.$ Suppose $f|_J$ is injective, $f^\prime(p)\not=0,$ for $p\in J,$ and the sets $f(J)$  and $C(f,J)$ are disjoint. Then, $f$ is injective (one-to-one conformal) on some neighborhood of $J.$
\end{theorem}

\begin{proof}
We may assume that $U$ is simply connected and so it is conformally equivalent to $\overline\C,$ to $\C$ or to the unit disc. The case $\overline\C$ is excluded, since $f$ is not constant. The conclusion now follows from the previous theorem.
\end{proof}

We shall now extend our results on arcs to arbitrary sets. Of course, a necessary condition that a function be extendable  conformally is that it be extendable holomorphically. To this end we have the following.

\begin{theorem}\label{holomorphic extension}
Let $X$ be a Riemann surface, $E$ an arbitrary subset  of $X$ and $f:E\rightarrow Y$  a mapping from $E$ to a complex manifold $Y,$ such that, for each $p\in E,$ there is an open neighborhood $U_p\subset X$ of $p$ and a holomorphic mapping $\Phi_p:U_p\rightarrow Y,$ such that $\Phi_p(q)=f(q),$ for all $q\in U_p\cap E.$
Then, $f$ extends to a holomorphic mapping $\Phi:U\rightarrow Y$ on some neighborhood $U$ of $E,$ such that $\Phi$ locally coincides with some $\Phi_p.$
\end{theorem}

A form of this theorem was proved in \cite[Th. 5]{G} for the special case of meromorphic functions, that is, when $Y=\overline\C.$ The theorem in \cite{G} is also weaker in the sense that it is not claimed that,  $\Phi$ locally coincides with some $\Phi_p$  but merely that $\Phi(p)=f(p).$ 

\begin{proof} The proof is a modification of an argument taken from \cite{G}. 
Let $E^\prime$ be the set of accumulation points of $E$ which are in $E.$  Choose a distance function on $X$ (see \cite{P}). For $p\in E,$ denote by $r_p$ the distance of $p$ from $\partial U_p.$ 
For each $p\in E^\prime,$ we choose a parametric disc $D_p$ for $X$ at $p,$ such that diam$D_p<r_p/2.$

Claim: for every two such discs $D_p, D_q,$ with $p,q\in E^\prime,$ we have
$$
    \Phi_p(z) = \Phi_q(z), \quad \mbox{for all} \quad z\in D_p\cap D_q.
$$
We may suppose that $D_p\cap D_q\not=\emptyset$ and\\ diam$D_q\le$ diam$D_p.$  Then, $D_q\subset U_p.$
Since $q$ is a limit point of $E,$ and both $\Phi_p$ and $\Phi_q$ equal $f$ on $E\cap D_q,$ it follows that $\Phi_p=\Phi_q$ on the 
component of $U_p\cap U_q$ containing $D_q.$ Obviously, this component contains $D_p\cap D_q$ so the claim follows. 

We have
$$
    E^\prime \subset U^\prime\stackrel{def}=\bigcup_{p\in E^\prime}D_p.
$$
By the claim, we may define a holomorphic function $\Phi$ on the open neighborhood $U^\prime$ of $E^\prime,$ by setting $\Phi=\Phi_p$ on each $D_p, p\in E^\prime.$ Moreover, $\Phi=f$ on $U^\prime\cap E.$

Arrange the points of $E\setminus U^\prime$ in a sequence $p_n.$
Denote by $U_p$ the neighborhood of $p.$ For each $p=p_n,$ choose a  disc $D_n=D_p$ centered at $p$ and contained in $U_p$ so small that the  radius is less than $1/n,$ and such that $E\cap\overline D_p=\{p\}.$ We can also arrange that the discs $D_n$ are pairwise disjoint. For instance, it suffices that the radius of $D_n$ be smaller than $1/2$ the distance of $p_n$ to the rest of $E.$ Set  $\Phi=\Phi_{p_n}$ on $D_n.$ Let $U$ be defined as
$$
    U = [U^\prime\setminus\cup_{p\in E\setminus U^\prime}\overline D_p]\cup\bigcup_{p\in E\setminus U^\prime}D_p.
$$
In order to check that the set U is open  one can use the fact that the
radii of the $D_n$ converge to zero. 
The mapping $\Phi$ is well defined on the neighborhood $U$ of $E$ and has the desired properties.
\end{proof}

Remark 1. In the proof of Theorem 3 the function $\Phi$ locally coincides
with some $\Phi_p.$ Therefore, if we assume that the derivative of $\Phi_p$ at $p$ is non zero, then we can consider smaller open sets $U_p$ so
that the derivative of $\Phi_p$ is everywhere non zero. It follows that
the derivative of $\Phi$ is non zero everywhere on $U$ and for every $z$ in $U$ the
mapping $\Phi$ is locally a homeomorphism between two open sets containing $z$ and $\Phi
(z) ,$ respectively.

A {\bf holomorphic curve} in a complex manifold $Y$ is a nonconstant holomorphic mapping $\Phi:X\rightarrow Y$  from a Riemann surface $X$ into $Y.$

\begin{corollary}
Let $X$ be a Riemann surface, $E$ a connected subset of $X$ and $f:E\rightarrow Y$  a mapping from $E$ to a complex manifold $Y,$ such that, for each $p\in E,$ there is an open neighborhood $U_p\subset X$ of $p$ and a holomorphic mapping $\Phi_p:U_p\rightarrow Y,$ such that $\Phi_p(q)=f(q),$ for all $q\in U_p\cap E.$
Then, $f$ extends to a holomorphic curve $\Phi:V\rightarrow Y$ mapping some connected neighborhood $V$ of $E$ into $Y,$  such that $\Phi$ locally coincides with some $\Phi_p.$ 
\end{corollary}

\begin{proof}
In Theorem \ref{holomorphic extension},  let $V$ be the component of $U$ containing $E.$ Then, $V$ is a Riemann surface and so $\Phi:V\rightarrow Y$ is, by definition, a holomorphic curve.
\end{proof}

The Corollary applies, in particular, to the case that $E$ is an open arc.

In order to extend a function, not only holomorphically, but even biholomorphically,  the following lemma \cite[Lemma 3.6]{M} is helpful. 

\begin{lemma} Let $U, Y$ be Hausdorff spaces with countable bases and $U$ be locally compact. 
If $\Phi : U \rightarrow Y$ is a local homeomorphism 
and the restriction of $\Phi$ to a closed subset $E$ is a homeomorphism, then $\Phi$
is a homeomorphism on some neighbourhood $V$ of $E.$
\end{lemma}

Let $E$ be a subset of a Riemann surface $X,$ and let $f:E\rightarrow Y.$  For a point $p\in X,$ we define the cluster set $C(f,p)$ of $f$ at $p$ as the set of all values $w\in Y,$ such that there is a sequence $z_n\in E, z_n\rightarrow p,$ for which $f(z_n)\rightarrow w.$

\begin{theorem}\label{X to Y} Let $X$ and $Y$ be Riemann surfaces and $E$ be an arbitrary subset  of $X.$
Suppose, for a function $f:E\rightarrow Y,$ that the cluster sets $C(f,p), p\in X,$  are pairwise disjoint and, for each $p\in E,$ there is an open neighborhood $U_p\subset X$ of $p$ and a holomorphic mapping $\Phi_p:U_p\rightarrow Y,$ such that $\Phi_p(q)=f(q),$ for all $q\in U_p\cap E$ and $\Phi_p^\prime(p)\not=0.$   Then $f$ extends to a  one-to-one conformal mapping  of  some open neighborhood $V$ of $E$ onto an open subset of $Y.$
\end{theorem}

\begin{proof} By Theorem 3, $f$ extends to a holomorphic mapping
into $Y.$ According to Remark 1, for every $z$ in $U,$ the mapping $\Phi$ is locally a
homeomorphism between two open sets containing $z$ and $\Phi(z),$ respectively.
Set $F=f(E).$ We claim that $f:E\rightarrow F$ is a homeomorphism. First of all, $f$ is continuous, because it is the restriction of the continuous function $\Phi.$ The continuity of $f$ implies that $C(f,z)=f(z),$ for all $z\in E.$ The hypothesis on cluster sets therefore implies that $f$ is injective and hence has an inverse function $\psi:F\rightarrow E.$ We claim that $\psi$ is continuous. 
To see this, let $b=f(a)$ be a point of $F$ and suppose $w_n=f(z_n)$ converges to $b.$ Let $a^\prime$ be any limit point of $\psi(w_n)=z_n.$ Then, from the definition of $C(f,a^\prime),$ it follows that $b\in C(f,a^\prime).$ Since $b$ is also in $C(f,a),$ the hypothesis on cluster sets implies that $a^\prime=a.$ We have shown that $\psi(w_n)\rightarrow a=\psi(b).$ This confirms the claim that $\psi$ is continuous and also that $f$ is a homeomorphism.

If $E$ is closed, the theorem now follows from Lemma 2. 

In general, fix some distance function on $Y.$
For each $z\in E,$ we may choose an open neighborhood $\widetilde U_z$ such that $\Phi(\widetilde U_z)$ is a disc $D_w$ centered at $w=f(z)$ and we may choose a branch $\Psi_w$ of $\Phi^{-1}$ in $D_w$ such that $\Psi_w\circ\Phi$ is the identity on $\widetilde U_z.$
We claim that $\Psi_w=\psi$ on $F\cap D_w.$ To verify the last claim, suppose $b\in F\cap D_w.$ From the definition of $\Psi_w,$ there is a point $a\in E\cap \widetilde U_z,$ such that $f(a)=b$ and $a=\Psi_w(b).$   Since $f$ is injective, $a=f^{-1}(b)=\psi(b).$ We have shown that $\Psi_w(b)=a=\psi(b),$ which establishes the claim.  By Theorem \ref{holomorphic extension}, there is an open neighborhood $W$ of $F$ and a holomorphic mapping $\Psi:W\rightarrow X,$ such that $\Psi=\Psi_w$ on $D_w,$ for each $w\in F.$ 
The mapping $\Psi$ is holomorphic on $W$ and maps $W$ to an open neighborhood $V$ of $E,$ contained in the domain of definition of $\Phi.$ Moreover, $\Psi\circ\Phi$ is the identity on $V.$ Thus, $\Phi$ is biholomorphic from $V$ to $W$ and $\Phi$ restricted to $E$ is $f.$ This finishes the proof.
\end{proof}

We remark that if $E$ has no isolated points, then the extension in Theorem 4 is unique in the sense that any two extensions agree on some neighborhood of $E.$ In particular, this applies to the case that $E$ is a curve.

\begin{definition}\label{analytic}[analytic arc]
An analytic open arc $J$ in a Riemann surface $X$ is an arc in $X,$ whose parametrization $\varphi$ is analytic. Thus, for every $t_0\in I=(0,1)$ and local coordinate $z$ in a neighborhood of $\varphi(t_0),$ the function $z\circ\varphi$ has a representation as a power series near $t_0:$  $(z\circ\varphi)(t)=\sum a_j(t-t_0)^j.$ We shall say that $J$ is a regular open analytic arc if $\varphi^\prime(t)\not=0,$ for all $t\in(0,1).$
\end{definition}

\begin{definition}\label{conformal}[conformal arc]
A conformal open arc $J$ in a Riemann surface is an arc in $X,$ with a parametrization which extends to a one-to-one conformal mapping from a neighborhood of $I\subset\C$ into $X.$
\end{definition}

\begin{corollary}
An open arc $J$ in a Riemann surface $X$ is a regular analytic arc if and only if it is a conformal arc.
\end{corollary}

\begin{proof}
Clearly, if $J$ is a conformal arc, then it is a regular analytic arc.

Suppose, conversely, that $J$ is a regular analytic arc. Consider firstly the case that $X=\overline\C.$
Thus, there is a parametrization $\varphi:I\rightarrow J$ which is analytic and such that $\varphi^\prime(t)\not=0,$ for $t\in I.$  For each $t\in I,$ we may extend $\varphi$ to a holomorphic function $\varphi_t$ in a disc $D_t$ centered at $t$ and disjoint from $\{0,1\}.$ If $D_s\cap D_t\not=\emptyset,$ then $\varphi_s=\varphi=\varphi_t$ on $I\cap D_s\cap D_t.$ Thus, $\varphi_s=\varphi_t$ on $D_s\cap D_t.$ By setting $\varphi=\varphi_t$ on $D_t,$ for each $t\in I,$ we obtain a well-defined holomorphic extension of $\varphi$ to the open neighborhood $U=\cup_{t\in I}D_t$ of $I.$

We note that the initial and terminal ends of $J$ are the same as the initial and terminal cluster sets $C_o(\varphi,I)$ and $C_1(\varphi,I).$  Since $\varphi:I\rightarrow J$ is a homeomorphism, these cluster sets are disjoint from $J.$ It follows directly from Theorem \ref{X to Y} that $J$ is a conformal arc. This completes the proof of the converse, in case $X=\overline\C.$

In the general case, it follows from Lemma 1, that there is a simply connected neighborhood $U$ of $J$ in $X.$ By the Riemann mapping theorem, $U$ is conformally equivalent to an open subset of $\C,$ and so it follows from the case just treated that $J$ is a conformal arc.
\end{proof}

This equivalence, stated in the corollary, is probably folkloric, at least in case the analytic arc is the image of the {\em closed} rather than open unit interval or in case we have a Jordan curve. Osserman \cite{O}  says that an analytic Jordan curve
on a Riemann surface $X$ is defined by an analytic mapping of the unit circle into $X$ and this mapping extends to a conformal mapping of an annulus into $X.$ In these cases, one can use compactness to give a simpler proof than the one we have given.

{\it We thank Lee Stout for bringing Lemma 2 to our attention.}

\end{document}